\documentclass[a4paper,nostix]{birthday}

\usepackage{graphicx}
\usepackage{enumitem}

\newtheorem{proposition}{Proposition}[section]
\newtheorem{theorem}[proposition]{Theorem}
\newtheorem{lemma}[proposition]{Lemma}

\newtheorem{definition}[proposition]{Definition}

\newenvironment{proofof}[1]{\smallskip\noindent{\textbf{Proof~of~#1.}}%
  \hspace{1pt}}{\hspace{-5pt}{\nobreak\quad\nobreak\hfill\nobreak%
    $\square$\vspace{2pt}\par}\smallskip\goodbreak}

\newenvironment{proof}{\smallskip\noindent{\textbf{Proof.}}%
  \hspace{1pt}}{\hspace{-5pt}{\nobreak\quad\nobreak\hfill\nobreak%
    $\square$\vspace{2pt}\par}\smallskip\goodbreak}

\numberwithin{equation}{section}
\allowdisplaybreaks[4]

\renewcommand{\phi}{\varphi}
\renewcommand{\theta}{\vartheta}
\renewcommand{\epsilon}{\varepsilon}
\renewcommand{\L}[1]{\mathbf{L^#1}}
\renewcommand{\d}[1]{\mathinner{\mathrm{d}{#1}}}

\newcommand{\C}[1]{\mathbf{C^{#1}}}
\renewcommand{\H}[1]{\mathbf{H^{#1}}}

\newcommand{\Cc}[1]{\mathbf{C_c^{#1}}}
\newcommand{\W}[2]{\mathbf{W^{#1,#2}}}
\newcommand{\BV}{\mathbf{BV}}
\newcommand{\modulo}[1]{{\left|#1\right|}}
\newcommand{\norma}[1]{{\left\|#1\right\|}}
\newcommand{\reali}{{\mathbb{R}}}
\newcommand{\naturali}{{\mathbb{N}}}

\DeclareMathOperator{\grad}{\nabla}

\DeclareMathOperator{\sign}{sign}

\DeclareMathOperator{\tr}{tr}
\DeclareMathOperator{\tv}{TV}
\let\div\relax
\DeclareMathOperator{\div}{\nabla\cdot}

\title{Modeling Crowd Dynamics through Hyperbolic -- Elliptic
  Equations}

\author{Rinaldo M.~Colombo$^1$, Maria Gokieli$^2$ and
  Massimiliano D.~Rosini$^3$}

\contact[rinaldo.colombo@unibs.it]{INDAM Unit, c/o DII, University of
  Brescia, Via Branze 38, 25123 Brescia, Italy}

\contact[M.Gokieli@icm.edu.pl]{ICM, University of Warsaw, ul.~Prosta
  69, P.O. Box 00-838, Warsaw, Poland}

\contact[ mrosini@umcs.lublin.pl]{Instytut Matematyki, Uniwersytet
  Marii Curie-Sk\l odowskiej, Plac Marii Curie-Sk\l odowskiej 1,
  20-031 Lublin, Poland}
  
\usepackage{calc}
\begin{document}

\begin{abstract}
  \noindent Inspired by the works of Hughes~\cite{Hughes2002,
    Hughes2003}, we formalize and prove the well posedness of a
  hyperbolic--elliptic system whose solutions describe the dynamics of
  a moving crowd. The resulting model is here shown to be well posed
  and the time of evacuation from a bounded environment is proved to
  be finite. This model also provides a microscopic description of the
  individuals' behaviors.
\end{abstract}



\section{Introduction}

We consider the problem of describing how pedestrians exit an
environment. From a macroscopic point of view, we identify the crowd
through the pedestrians' density, say $\rho = \rho (t,x)$, and assume
that the crowd behavior is well described by the continuity equation
\begin{displaymath}
  \partial_t \rho + \div \left( \rho \; V (x, \rho)\right) = 0 \;,
  \qquad
  (t,x) \in \reali^+ \times \Omega\;,
\end{displaymath}
where $\Omega \subset \reali^2$ is the environment available to
pedestrians, $V = V (x,\rho) \in \reali^2$ is the velocity of the
individual at $x$, given the presence of the density $\rho$. Several
choices for the velocity function are available in the literature, see
for instance~\cite{borscheColomboGaravelloMeurer2015,
  BrunoTosinTricerriVenuti2011, ColomboGaravelloMercier2012,
  ColomboMercier2012, CristianiPriuliTosin2015, Hughes2002,
  Hughes2003, JiangZhouTian2015, GoatinTwarogowskaDuvigneau2014} for
velocities depending nonlocally on the density,
and~\cite[Section~4.1]{kachroo2009} for velocities depending locally
on the density. Here, we posit the following (local with respect to
the density) assumption:
\begin{displaymath}
  V (x,\rho) = v (\rho) \; w (x)
\end{displaymath}
where $v = v (\rho)$ is a smooth non-increasing scalar function,
motivated by the common attitude of moving faster when the density is
lower. A key role is played by $w = w (x)$: this vector identifies the
route followed by the individual at $x$. It is reasonable to assume
that the individual at $x$ follows the shortest path from $x$ towards
the nearest exit. This naturally suggests to choose $w$ parallel to
$\nabla \phi$, the potential $\phi$ being the solution to the eikonal
equation on $\Omega$. Extending the results
in~\cite{AmadoriGoatinRosini, 1d, ElKhatibGoatinRosini2013} obtained
in the 1-dimensional space to the 2-dimensional space, we consider the
following elliptic regularization of the eikonal equation:
\begin{displaymath}
  \norma{\grad \phi}^2 - \delta \; \Delta \phi = 1 \;,
  \qquad
  x\in\Omega\;,
\end{displaymath}
where $\delta$ is a fixed strictly positive parameter. Clearly, the
resulting vector field $\grad \! \phi$ depends only on $\Omega$,
namely only on the geometry of the environment available to the
pedestrians, i.e., on the positions of the exits, on the possible
presence of obstacles, and so on. We assume that the boundary
$\partial\Omega$ is partitioned in walls, say $\Gamma_w$, exits, say
$\Gamma_e$, and corners, say $\Gamma_c$; namely
$\partial \Omega = \Gamma_w \cup \Gamma_e \cup \Gamma_c$, the set
$\Gamma_e$, $\Gamma_w$, $\Gamma_c$ being two by two disjoint.
$\Gamma_c$ is a discrete subset of $\partial\Omega$. Also $\Gamma_e$
and $\Gamma_w$ are subsets of $\partial\Omega$ and they are open in
the topology they inherit from $\partial \Omega$. It is then natural
to choose $\phi$ as solution to the elliptic equation
\begin{equation}
  \label{eq:ell}
  \left\{
    \begin{array}{l@{\qquad}r@{\;}c@{\;}l}
      \norma{\grad \phi}^2 -  \delta \; \Delta \phi = 1
      & x & \in & \Omega
      \\
      \grad \phi (\xi) \cdot \nu (\xi) = 0
      & \xi & \in & \Gamma_w
      \\
      \phi (\xi) = 0
      & \xi & \in & \Gamma_e \;,
    \end{array}
  \right.
\end{equation}
$\nu (\xi)$ being the outward unit normal to $\partial\Omega$ at
$\xi$.  To select the direction $w (x)$ followed by the pedestrian at
$x$ we set
\begin{equation}
  \label{eq:state}
  w = \mathcal{N} (-\nabla\phi) \;,
\end{equation}
the map $\mathcal{N}$ being a regularized normalization, that is
\begin{equation}
  \label{eq:N}
  \mathcal{N} (x) = \frac{x}{\sqrt{\theta^2+\norma{x}^2}}\;,
\end{equation}
for a fixed strictly positive parameter $\theta$. Finally, the
evolution of the crowd density $\rho$ is then found solving the
following scalar conservation law:
\begin{equation}
  \label{eq:hyp}
  \left\{
    \begin{array}{l@{\qquad}r@{\;}c@{\;}l}
      \partial_t \rho + \div \left( \rho \; v (\rho) \; w (x) \right) = 0
      & (t,x) & \in & \reali^+ \times \Omega
      \\
      \rho (0,x) = \rho_o (x)
      & x & \in & \Omega
      \\
      \rho (t, \xi) = 0
      & (t, \xi) & \in & \reali^+ \times \partial\Omega \;,
    \end{array}
  \right.
\end{equation}
where $\rho_o$ is the initial crowd distribution. In other words, for
a given domain $\Omega$, from~\eqref{eq:ell} we obtain the vector
field $\nabla\phi$, that is used in~\eqref{eq:state} to define $w$ and
then from~\eqref{eq:hyp} we obtain how the pedestrians' density $\rho$
evolves in time starting from the initial density $\rho_o$.

Remark that the boundary condition $\rho (t, \xi) = 0$ has to be
understood in the sense of conservation laws,
see~\cite{BardosLerouxNedelec, ColomboRossi2015} and
Definition~\ref{def:21} below. Indeed, the choice in~\eqref{eq:hyp}
allows a positive outflow from $\Omega$ through $\Gamma_e$ thanks to
the definition of $w$, as proved in~\ref{prop:1.2} of
Proposition~\ref{prop:1}.

We prove below that the model consisting
of~\eqref{eq:ell}--\eqref{eq:state}--\eqref{eq:hyp} is well posed,
i.e., it admits a unique solution which is a continuous function of
the initial data. Moreover, we also ensure that the evacuation time is
finite.

Remark that the model~\eqref{eq:ell}--\eqref{eq:state}--\eqref{eq:hyp}
is completely defined by the physical domain $\Omega$, by the function
$v = v (\rho)$ and by the initial datum $\rho_o$, apart from the
regularizing parameters $\delta$ and $\theta$.

\smallskip

The next two sections are devoted to the detailed formulation of the
problem, to the statement of the well posedness result and of further
qualitative properties of the
model~\eqref{eq:ell}--\eqref{eq:state}--\eqref{eq:hyp}. All technical
details are gathered in Section~\ref{sec:TD}.

\section{Well Posedness}
\label{sec:Main}

Throughout, we denote $\reali^+ = \left[0, \infty\right[$. For
$x \in \reali^2$ and $r>0$, $B (x,r)$ stands for the open disk
centered at $x$ with radius $r$. For any measurable subset $S$ of
$\reali^2$, we denote by $\modulo{S}$ its 2-dimensional Lebesgue
measure. Recall that two (non-empty) subsets $A_1$, $A_2$ of
$\reali^2$ are \emph{separate} whenever
$\overline{A_1} \cap A_2 = \emptyset = A_1 \cap \overline{A_2}$.

A key role is played by the geometry of the domain $\Omega$. Here we
collect the conditions necessary in the sequel, see
Figure~\ref{fig:Omega}.

\begin{enumerate}[label={\textbf{($\mathbf\Omega$.\arabic*)}},
  leftmargin=*]
\item\label{omega1} $\Omega \subset \reali^2$ is non-empty, open,
  bounded and connected.
\item\label{omega2} The boundary $\partial\Omega$ admits the disjoint
  decomposition
  $\partial \Omega = \Gamma_w \cup \Gamma_e \cup \Gamma_c$, where
  $\Gamma_w$ and $\Gamma_e$ are separate and are finite union of open
  1-dimensional manifolds of class $\C{3,\gamma}$, for a given
  $\gamma \in \left]0, 1 \right[$; $\Gamma_e$ is non-empty;
  $\Gamma_c $ is a discrete finite set and
  $\overline{\Gamma_w} \cap \overline{\Gamma_e} \subseteq \Gamma_c
  \subseteq \overline{\Gamma_w}$.
\item\label{omega3} For any $x\in \Gamma_c$, there exists an
  $\epsilon > 0$ such that the intersection $B(x,\epsilon)\cap\Omega$
  is exactly a quadrant of the disk $B(x,\epsilon)$.
\end{enumerate}
\begin{figure}[!h]
  \centering
  \includegraphics[width=.5\textwidth]{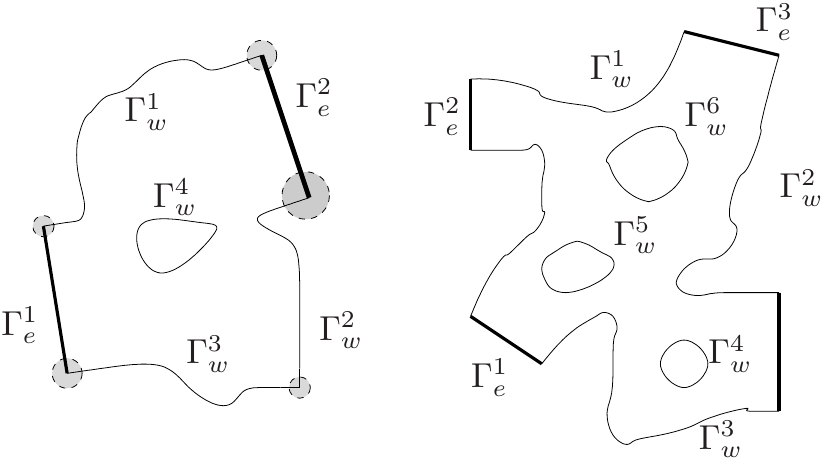}
  \caption{Two examples of sets $\Omega$ with the notation used
    in~\ref{omega2} and in~\ref{omega3}.\label{fig:Omega}}
\end{figure}

\noindent The requirement~\ref{omega1} is clear. In~\ref{omega2}, the
term \emph{open} has to be understood with respect to the topology
inherited by $\partial\Omega$. Again concerning~\ref{omega2},
introduce the connected components of $\Gamma_w, \Gamma_e$ and
$\Gamma_c$, i.e.,
\begin{displaymath}
  \Gamma_w = \bigcup_{i=1}^{n_w} \Gamma_w^i \;,
  \qquad
  \Gamma_e = \bigcup_{i=1}^{n_e} \Gamma_e^i \;,
  \quad \mbox{ and } \quad
  \Gamma_c = \bigcup_{i=1}^{n_c} \{J_i\} \;.
\end{displaymath}
Each of the $\Gamma_e^i$ is an exit, while the $J_i$ are points where
the regularity of $\partial\Omega$ is allowed to be
lower. Condition~\ref{omega2} implies that each $\Gamma_w^i$ and each
$\Gamma_e^i$ is a $\C{3,\gamma}$ manifold. Since
$\Gamma_c \subseteq \overline{\Gamma_w}$, along the boundary
$\partial\Omega$, between two different exits there is always a wall
or, in other words, there can not be two exits separated only by a
corner point. Condition~\ref{omega2} also implies that $n_e\geq 1$, so
that there is at least one exit. Moreover, apart from the trivial case
where $\partial \Omega = \Gamma_e$, the set $\Gamma_c$ may not be
empty. Note also that any corner point $J_i$ in $\Gamma_c$ is either a
doorjamb, if $J_i \in \overline{\Gamma_e}$, or a wall corner, if
$J_i \in (\overline{\Gamma_w} \setminus \overline{\Gamma_e})$.
Condition~\ref{omega3} says that the angles between each door and the
walls are right and convex, and additionally that these contain
straight segments. This is a technical assumption, related to the
subtle mixed boundary conditions: Dirichlet and Neumann conditions
meet at the doorjamb points. Condition~\ref{omega3} ensures the
regularity of solutions in a neighborhood of these points, a property
that might not hold for general angles.

Throughout, by \emph{solution} to~\eqref{eq:ell} we mean
\emph{generalized solution} in the sense of the following definition
(see~\cite[Chapters~8 and~13]{gilbarg}).

\begin{definition}
  \label{def:phi}
  Let $\Omega$ satisfy~\ref{omega1}. \!A function
  $\!\phi \!\in\! \H1\!(\Omega;\reali)\!$ is a \emph{generalized solution}
  to~\eqref{eq:ell} if $\tr_{\strut\Gamma_e} \!\!\!\phi \!=\! 0$ and
  \begin{displaymath}
    \delta\int_{\Omega}\nabla \phi (x) \cdot \nabla\eta (x) \d{x} +
    \int_{\Omega} \left(\norma{\nabla\phi (x)}^2 - 1\right) \, \eta (x)
    \d{x} = 0
  \end{displaymath}
  for any $\eta \in \H1(\Omega;\reali)$ such that
  $\tr_{\strut\Gamma_e} \eta = 0$.
\end{definition}

\noindent Above, $\tr_{\strut\Gamma_e} \eta$ denotes the trace of
$\eta$ on $\Gamma_e$. We refer to~\cite[Chapter~5.5]{Evans} for the
definition and properties of the trace operator.

Note that no generalized solution to~\eqref{eq:ell} can vanish a.e.~on
$\Omega$. The next proposition provides the basic existence result for
the solutions to~\eqref{eq:ell}, together with some qualitative
properties.

\begin{proposition}[Elliptic Problem]
  \label{prop:1}
  Let $\Omega$ satisfy~\ref{omega1}, \ref{omega2}, \ref{omega3}.  Fix
  $\delta>0$. Then, problem~\eqref{eq:ell} admits a unique generalized
  solution $\phi \in \C3 (\overline\Omega; \reali)$ with the
  properties:
  \begin{enumerate}[label={\bf(E.\arabic*)}, leftmargin=*]
  \item\label{prop:1.1} For a.e.~$x \in \Omega$,
    $\grad \phi (x) \neq 0$.
  \item\label{prop:1.2} For all $\xi \in \Gamma_e $,
    $- \grad\phi{(\xi{})} \cdot \nu{(\xi{})} > 0$.
  \item\label{prop:1.3}
    $\displaystyle \frac{\modulo{\Omega}}{\delta} \; \exp\left(
      -\frac{\max_{\partial\Omega}\phi}{\delta}\right) \leq
    -\int_{\Gamma_e} \grad\phi{(\xi)} \cdot \nu{(\xi)} \d\xi \leq
    \frac{\modulo{\Omega}}{\delta} \; \exp\left(
      \frac{\max_{\partial\Omega}\phi}{\delta}\right)$.
  \end{enumerate}
\end{proposition}

\noindent The proof of the above proposition is postponed to
Section~\ref{sec:TD}. Here, we note that properties~\ref{prop:1.1},
\ref{prop:1.2} and~\ref{prop:1.3} have clear consequences on the
properties of the solutions to the full
system~\eqref{eq:ell}--\eqref{eq:state}--\eqref{eq:hyp}. Indeed,
setting $w$ as in~\eqref{eq:state}, property~\ref{prop:1.1} implies
that $w$ vanishes only on a set of measure $0$; \ref{prop:1.2} ensures
that $w$ is non zero and points outwards along exits; \ref{prop:1.3}
can be used to provide bounds on the evacuation time.

In the hyperbolic problem~\eqref{eq:hyp}, we use the following
assumptions, which are standard in the framework of conservation laws:
\begin{enumerate}[label={\textbf{(C.\arabic*)}},leftmargin=*]
\item\label{V} $v \in \C2 ([0,R_{\max}]; [0, V_{\max}])$ is weakly
  decreasing, $v(0) = V_{\max}$ and $v(R_{\max}) = 0$.
\item\label{C}
  $\rho_o \in (\BV \cap \L\infty) (\Omega; [0,R_{\max}])$.
\end{enumerate}
Above, $R_{\max}$, respectively $V_{\max}$, is the maximal density,
respectively speed, possibly reached by the pedestrians.

We recall also the definition of entropy solution to~\eqref{eq:hyp},
which originates in~\cite{Volpert}, see
also~\cite[p.~1028]{BardosLerouxNedelec}. Here, we refer
to~\cite[Definition~2.1]{ColomboRossi2015}.

\begin{definition}
  \label{def:21}
  Let the conditions~\ref{omega1}, \ref{omega2}, \ref{V} and~\ref{C}
  hold. Let $w \in \C2 (\overline{\Omega}; \overline{B}(0,1))$. A
  function
  $\rho \in (\L\infty \cap \BV) ([0,T] \times \Omega;[0,R_{\max}])$ is
  an \emph{entropy solution} to the initial -- boundary value
  problem~\eqref{eq:hyp} if for any test function
  $\zeta\in\Cc{2}( ]-\infty, T[\times \reali^2; \reali^+)$ and for any
  $k \in [0,R_{\max}]$
  \begin{eqnarray*}
    \int_0^T \int_\Omega
    \big\{
    \modulo{\rho(t,x)-k} \, \partial_t\zeta(t,x)
    +
    \sign \left( \rho(t,x) - k\right)
    \left(\rho(t,x) \, v\left(\rho(t,x)\right) - k \, v (k)\right) w(x)
    \cdot \grad\zeta(t,x)
    \big\} \d{x} \, \d{t}
    \\
    +
    \int_\Omega \modulo{\rho_o(x) - k} \; \zeta(0,x) \; \d{x}
    +
    \int_0^T \int_{\partial\Omega}
    \left(
    {\tr_{\strut\partial\Omega}\rho}(t,\xi) \,
    v\left({\tr_{\strut\partial\Omega}\rho}(t,\xi)\right) - k \, v(k)
    \right) w(\xi)
    \cdot \nu(\xi) \, \zeta(t,\xi) \, \d{\xi} \, \d{t}
    &\geq
    0 \,.
  \end{eqnarray*}
\end{definition}

\noindent As above, $\tr_{\strut\partial\Omega} u$ stands for the
operator trace at $\partial \Omega$ applied to the $\BV$ function $u$,
see for instance~\cite[\S~5.5]{Evans}
or~\cite[Appendix]{ColomboRossi2015}. Note that if the solution has
bounded total variation in time, it has a trace at $t=0+$.

\begin{proposition}[Hyperbolic Problem]
  \label{prop:3}
  Let the conditions~\ref{omega1}, \ref{omega2} and~\ref{V} hold. Let
  $w \in \C2 (\overline{\Omega}; \overline{B}(0,1))$. Then,
  problem~\eqref{eq:hyp} generates the map
  \begin{displaymath}
    \begin{array}{ccc@{}c@{}ccc}
      \mathcal{S} & \colon
      & \reali^+ & \times & (\L1\cap \BV) (\Omega; [0,R_{\max}]) & \to
      & (\L1\cap \BV)(\Omega; [0,R_{\max}])
      \\
                  & & t & , & \rho & \mapsto & \mathcal{S}_t \rho
    \end{array}
  \end{displaymath}
  with the following properties:
  \begin{enumerate}[label={\bf(H.\arabic*)}, leftmargin=*]
  \item\label{prop:3.1} $\mathcal{S}$ is a semigroup.
  \item\label{prop:3.2} $\mathcal{S}$ is Lipschitz continuous with
    respect to the $\L1$-norm, more precisely for any $s,t\in[0,T]$
    \begin{displaymath}
      \norma{
        \mathcal{S}_t\rho_o
        -
        \mathcal{S}_s\rho_o}_{\L\infty(\Omega;\reali)}
      \leq
      \left[ \sup_{\tau\in[s,t]} \tv(\mathcal{S}_\tau \rho_o) \right]
      \modulo{t-s} \;.
    \end{displaymath}
  \item\label{prop:3.3} For any $t \in [0,T]$
    \begin{align*}
          &\norma{\mathcal{S}_t\rho_o}_{\L\infty(\Omega;\reali)}
          \leq
               \norma{\rho_o}_{\L\infty(\Omega;\reali)} \,
               \exp(C_1 \, t) \;,
      &
      \tv\left(\mathcal{S}_t\rho_o\right)
          \leq
               C_2 \left(1+t+\tv(\rho_o)\right) \exp (C_2 \, t) \;,
    \end{align*}
    where the constants $C_1, C_2$ depend only on $R_{\max}$,
    $\norma{v'}_{\W2\infty ([0,R_{\max}];\reali)}$ and
    $\norma{w}_{\W2\infty (\Omega;\reali^2)}$.
  \item\label{prop:3.4} For any
    $\rho_o \in (\L1\cap \BV) (\Omega; [0,R_{\max}])$, the orbit
    $t \mapsto \mathcal{S}_t \rho_o$ is the unique solution
    to~\eqref{eq:hyp} in the sense of Definition~\ref{def:21}.
  \end{enumerate}
\end{proposition}

\noindent The proof of the above proposition is deferred to
Section~\ref{sec:TD}, where it is shown that the above statements
follow from~\cite[Theorem~2.7]{ColomboRossi2015}.

We now give the definition of solution
to~\eqref{eq:ell}--\eqref{eq:state}--\eqref{eq:hyp}.

\begin{definition}\label{def:ellhyp}
  Let the assumptions~\ref{omega1}, \ref{omega2}, \ref{omega3},
  \ref{V} and~\ref{C} hold. The pair of functions
  $(\phi,\rho) \in \H1(\Omega;\reali) \times (\L\infty \cap \BV)
  ([0,T]\times\Omega;[0,R_{\max}])$
  solves the problem~\eqref{eq:ell}--\eqref{eq:state}--\eqref{eq:hyp} if
  $\phi$ is a generalized solution to~\eqref{eq:ell} in the sense of
  Definition~\ref{def:phi} and $\rho$ is an entropy solution
  to~\eqref{eq:hyp} in the sense of Definition~\ref{def:21} with $w$
  given by~\eqref{eq:state}.
\end{definition}

The next theorem ensures the well posedness of the
elliptic--hyperbolic
model~\eqref{eq:ell}--\eqref{eq:state}--\eqref{eq:hyp}.

\begin{theorem}[Mixed Problem]
  \label{thm:main}
  Let the conditions~\ref{omega1}, \ref{omega2}, \ref{omega3}, \ref{V}
  and~\ref{C} hold. For any $\delta, \theta>0$, the
  elliptic-hyperbolic
  problem~\eqref{eq:ell}--\eqref{eq:state}--\eqref{eq:hyp} generates a map
  \begin{displaymath}
    \begin{array}{ccc@{}c@{}ccc}
      \mathcal{M} & \colon
      & \reali^+ & \times & (\L1\cap \BV) (\Omega; [0,R_{\max}]) & \to
      & (\L1\cap \BV)(\Omega; [0,R_{\max}])
      \\
                  & & t & , & \rho & \mapsto & \mathcal{M}_t \rho
    \end{array}
  \end{displaymath}
  with the following properties:
  \begin{enumerate}[label={\bf(M.\arabic*)}, leftmargin=*]
  \item\label{thm:main.1} $\mathcal{M}$ is a semigroup.
  \item\label{thm:main.2} $\mathcal{M}$ is Lipschitz continuous with
    respect to the $\L1$-norm, more precisely for any $s, t \in [0,T]$
    \begin{displaymath}
      \norma{\mathcal{M}_t\rho_o -
        \mathcal{M}_s\rho_o}_{\L\infty(\Omega;\reali)} \leq
      \left[\sup_{\tau\in[s,t]}\tv(\mathcal{M}_\tau\rho_o)\right]
      \modulo{t-s} \;.
    \end{displaymath}
  \item\label{thm:main.3} For any $t \in [0,T]$ we have that
    $(\phi,\rho)=\mathcal{M}_t\rho_o$ satisfies
    \begin{align*}
      &\norma{\rho}_{\L\infty(\Omega;\reali)}
        \leq
               \norma{\rho_o}_{\L\infty(\Omega;\reali)} \,
               \exp (C_1 \, t) \;,
      &\tv\left(\rho\right)
        \leq 
               C_2 \left(1+t+\tv(\rho_o)\right) \exp(C_2 \, t) \;,
    \end{align*}
    where $C_1$ is a positive constant depending on
    $\norma{q}_{\W{1}\infty([0,R_{\max}];\reali)}$ and
    $\norma{w}_{\W{1}\infty(\Omega;\reali^2)}$, while the constant
    $C_2$ depends on $\norma{q}_{\W{2}\infty([0,R_{\max}];\reali)}$
    and $\norma{w}_{\W{2}\infty(\Omega;\reali^2)}$, where as usual we
    set $q (\rho) = \rho \, v (\rho)$.
  \item\label{thm:main.4} For all
    $\rho_o \in (\L1\cap \BV) (\Omega; [0,R_{\max}])$, the orbit
    $t \mapsto \mathcal{M}_t \rho_o$ is the unique solution
    to~\eqref{eq:ell}--\eqref{eq:state}--\eqref{eq:hyp} in the sense
    of Definition~\ref{def:ellhyp}.
  \end{enumerate}
\end{theorem}

\noindent The above result is a direct consequence of
Proposition~\ref{prop:1} and Proposition~\ref{prop:3}.

\section{Qualitative Properties}
\label{sec:Qual}

Here, we aim at further qualitative properties of the solutions
to~\eqref{eq:ell}--\eqref{eq:state}--\eqref{eq:hyp} that have a
relevant meaning in the present setting.

Introduce for $\hat x \in \Omega$ the path $p_{\hat x}$ followed by
those pedestrians that are at $\hat x$ at time $t = 0$, i.e., the map
$p_{\hat x}$ is defined for $t \geq 0$ as the solution to the Cauchy
problem
\begin{equation}
  \label{eq:pat}
  \left\{
    \begin{array}{l}
      \dot x
      =
      w (x)
      \\
      x (0) = \hat x \;,
    \end{array}
  \right.
  \quad \mbox{ where } \quad
  w = \mathcal{N} \left(- \nabla \phi \right) \,.
\end{equation}
Above, $\mathcal{N}$ is defined in~\eqref{eq:N} and $\phi$ is the
solution to~\eqref{eq:ell}.

\begin{proposition}[Pedestrians' Trajectories]
  \label{prop:path}
  Let $\Omega$ satisfy~\ref{omega1}, \ref{omega2}, \ref{omega3} and
  call $\phi$ the solution to~\eqref{eq:ell} provided by
  Proposition~\ref{prop:1}. Then:
  \begin{enumerate}[label={\bf{(Q.\arabic*)}},leftmargin=*]
  \item\label{quark1} For any $\hat x \in \Omega$, there exists a
    unique globally defined path
    $p_{\hat x} \colon I_{\hat x} \to \reali^2$
    solving~\eqref{eq:pat}, $I_{\hat x}$ being a suitable non trivial
    real interval.
  \item\label{quark2} Any two paths either coincide or do not
    intersect, in the sense that for any $\hat x, \hat y \in \Omega$
    \begin{displaymath}
      p_{\hat x} (I_{\hat x}) \cap p_{\hat y} (I_{\hat y})
      \neq \emptyset \implies
      \left\{
        \begin{array}{lc}
          \mbox{either }
          & \hat x \in p_{\hat y} (I_{\hat y})
            \mbox{ and }
            p_{\hat x} (I_{\hat x}) \subseteq p_{\hat y} (I_{\hat y})
          \\
          \mbox{or }
          & \hat y \in p_{\hat x} (I_{\hat x})
            \mbox{ and }
            p_{\hat y} (I_{\hat y}) \subseteq p_{\hat x} (I_{\hat x}) \;.
        \end{array}
      \right.
    \end{displaymath}
  \item\label{quark3} There exist a subset
    $\hat \Omega \subset \Omega$ with $\modulo{\hat\Omega} = 0$ and a
    map $T \colon \Omega \setminus \hat\Omega \to \reali^+$ such that
    $I_{\hat x} = [0, T_{\hat x}]$ and
    $p_{\hat x} (T_{\hat x}) \in \Gamma_e$ for all
    $x \in \Omega \setminus \hat \Omega$.
  \end{enumerate}
\end{proposition}

\noindent The proof is deferred to Section~\ref{sec:TD}. In other
words, $T_{\hat x}$ is the time that the pedestrian leaving from point
$\hat x$ needs to reach the exit. Property~\ref{quark3} ensures that
this time is finite for a.e.~initial position $\hat x$.
\begin{figure}[!h]
  \centering
  \includegraphics[width=.23\textwidth]{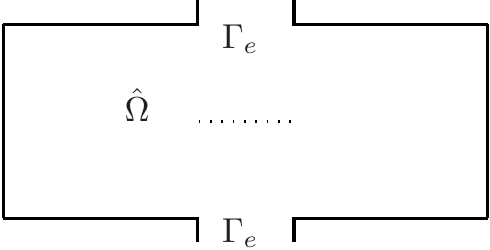}
  \caption{An example in which the set $\hat \Omega$ in
    Proposition~\ref{prop:path} is necessarily non empty. In the room
    above, due to the presence of the two exits $\Gamma_e$, the vector
    field $w$ vanishes along the dotted segment $\hat \Omega$.}
  \label{fig:buridan}
\end{figure}
Figure~\ref{fig:buridan} shows that the set $\hat \Omega$ may not be
avoided under the present assumptions.

\section{Technical Details}
\label{sec:TD}

We choose the following notation to denote a vector orthogonal to a
given vector in $\reali^2$:
\begin{displaymath}
  \mbox{if } v =
  \begin{bmatrix} v_1 \\ v_2 \end{bmatrix}
  \,, \quad \mbox{ then }
  v^\perp =
  \begin{bmatrix} -v_2\\ v_1 \end{bmatrix}
  \,.
\end{displaymath}
We frequently use the boundedness and Lipschitz continuity of the map
$\mathcal{N}$ as defined in~\eqref{eq:N}, namely
\begin{equation}
  \label{eq:Nprop}
  \begin{array}{rcll}
    \norma{{\cal N} (x)}
    & \leq
    & 1
    & \mbox{for all } x \in \reali^2\,,
    \\[2pt]
    \norma{{\cal N} (x_1) - {\cal N}(x_2)}
    & \leq
    & \theta^{-1} \; \norma{x_1 - x_2}
    &  \mbox{for all } x_1,x_2 \in \reali^2 \;.
  \end{array}
\end{equation}
The Hopf-Cole transformation (see e.g.~\cite[Chapter~4.4.1]{Evans})
\begin{equation}
  \label{eq:trans}
  u = e^{-\phi/\delta}
\end{equation}
transforms generalized solutions to~\eqref{eq:ell} into generalized
solutions to the \emph{linear} problem
\begin{equation}
  \label{eq:11}
  \left\{
    \begin{array}{l@{\qquad}r@{\;}c@{\;}l}
      u = \delta^2 \, \Delta u & x & \in & \Omega
      \\
      \grad u (\xi) \cdot \nu (\xi) = 0 & \xi & \in & \Gamma_w
      \\
      u (\xi) = 1 & \xi & \in & \Gamma_e \;,
    \end{array}
  \right.
\end{equation}
whse precise definition (see e.g.~\cite[Chapter 8]{gilbarg}) is here
below.

\begin{definition}
  \label{def:GenSol}
  A function $u\in \H1(\Omega;\reali)$ is a \emph{generalized
    solution} to~\eqref{eq:11} on $\Omega$ if
  $\tr_{\strut\Gamma_e} u \equiv 1$ and
  \begin{equation}
    \label{eq:EqUvar}
    \delta^2\int_{\Omega} \nabla u (x) \cdot \nabla\eta (x) \d{x}
    +
    \int_{\Omega} u (x) \, \eta (x) \d{x}
    =
    0
  \end{equation}
  for any $\eta \in \H1(\Omega;\reali)$ such that
  $\tr_{\Gamma_e}^{\,} \eta \equiv 0$.
\end{definition}

The next Lemma collects various information on~\eqref{eq:11}.

\begin{lemma}
  \label{lem:eu}
  Fix a positive $\delta$ and let $\Omega$ satisfy~\ref{omega1}
  and~\ref{omega2}. Then,
  \begin{enumerate}[label={\bf{(u.\arabic*)}}, leftmargin=*]
  \item\label{it:u1} Problem~\eqref{eq:11} admits a unique generalized
    solution $u \! \in \! (\H1 \cap \C\infty) (\Omega;\reali)$ in the
    sense of Definition~\ref{def:GenSol}. Moreover,
    $u \in \C3 (\overline{\Omega} \setminus \Gamma_c; \reali)$.
  \item\label{it:u2} There exists a positive $\varpi$ dependent only
    on $\Omega$ such that $u (x) \! \in \! \left]\varpi,1\right[$ for
    all $x \in \Omega$, so that $u (x) \in [\varpi,1]$ also for all
    $x \in \overline{\Omega}$.
  \item\label{it:u3} The solution $u$ to~\eqref{eq:11} satisfies
    $\nabla u(\xi)\cdot \nu (\xi) > 0$ for all $\xi \in \Gamma_e$.
  \item\label{it:u4} The set
    $\left\{x\in\Omega \colon \grad u (x) = 0 \right\}$ of critical
    points of $u$ has measure $0$.
  \end{enumerate}
  \noindent If in addition $\Omega$ satisfies~\ref{omega3}, then:
  \begin{enumerate}[label={\bf{(u.\arabic*)}}, leftmargin=*]
    \setcounter{enumi}{4}
  \item\label{it:u5} $u \in \C{3}(\overline\Omega;\reali)$.
  \item\label{it:u6} If $\bar{x}\in\overline{\Omega}$ is a critical
    point of $u$, then the Hessian matrix $D^2 u(\bar{x})$ has at
    least one positive eigenvalue.
  \end{enumerate}
\end{lemma}

\begin{proof}
  Consider the different items above separately.

  \noindent$\star$~\ref{it:u1}: we use Lax--Milgram Lemma,
  see~\cite[Section~6.2.1]{Evans}. Introduce the Hilbert space
  $H = \{\eta \in \H1(\Omega;\reali ) \colon \tr_{\strut\Gamma_e} \eta
  = 0\mbox{ a.e.~on } \Gamma_e \}$
  endowed with the usual scalar product and the coercive bilinear form
  \begin{displaymath}
    a(u,\eta)
    =
    \delta^2 \int_{\Omega}\nabla u (x)\cdot \nabla\eta (x) \d{x}
    +
    \int_{\Omega}u (x)\; \eta (x) \d{x} \;.
  \end{displaymath}
  Note that $H$ is a closed subspace of $\H1(\Omega;\reali)$ by the
  Trace Theorem~\cite[Chapter~5.5, Theorem~1]{Evans}. Indeed, if $u^k$
  is a sequence in $H$ converging to $u$ in $\H1(\Omega; \reali)$,
  then
  \begin{displaymath}
    \norma{u}_{\L2(\Gamma_e;\reali)}
    =
    \norma{u^k-u}_{\L2(\Gamma_e;\reali)}
    \leq
    C \, \norma{u^k-u}_{\H1(\Omega;\reali)} \to 0 \;,
  \end{displaymath}
  for a constant $C$ depending only on $\Omega$, so that $u \in H$. A
  function $u \in \H1 (\Omega;\reali)$ is a generalized solution
  to~\eqref{eq:11} if and only if $v = u-1\in H$ and
  $a(v,\eta) = - \int_\Omega \eta(x) \d{x}$ for all $\eta\in H$. The
  map $\eta \mapsto \int_\Omega \eta(x) \d{x}$ is a linear functional
  over $H$. By Lax--Milgram Lemma, we infer the existence and
  uniqueness of a generalized solution $u$ to~\eqref{eq:11} such that
  $u \in H \subset \H1 (\Omega;\reali)$.  Moreover,
  $u \in \C\infty(\Omega;\reali)$ by~\cite[Theorem~3 in Chapter~6.3
  and Theorem~6 in Section~5.6.3]{Evans}.  By~\ref{omega1}
  and~\ref{omega2}, the results in~\cite[Theorem~9.3]{agmon} ensure
  that $u \in \C3 (\overline{\Omega} \setminus \Gamma_c; \reali)$.

  \smallskip

  \noindent$\star$~\ref{it:u2}: note that, due to the boundary
  conditions along $\Gamma_e$ and $\Gamma_w$, no $\H1$ solution
  to~\eqref{eq:11} can be constant. The function $\eta = (u-1)^+$,
  where $(v)^+ = \max(v,0)$, is in $\H1 (\Omega;\reali)$ and inserting
  it in~\eqref{eq:EqUvar} we get
  \begin{displaymath}
    \delta^2 \int_{\Omega} \norma{\nabla( u-1)^+}^2 + \int_{\Omega}
    \modulo{(u-1)^+}^2 + \int_{\Omega} (u-1)^+ = 0 \;.
  \end{displaymath}
  This leads to $(u-1)^+ \equiv 0$ a.e.~in $\Omega$, and, by the
  continuity of $u$ on $\overline{\Omega}$, $u (x) \leq 1$ for all
  $x \in \overline{\Omega}$.  The map $u$ satisfies~\eqref{eq:11} in
  the strong sense everywhere in $\Omega$. Hence, by the maximum
  principle~\cite[Chapter~2, Theorem~6]{protter} $u (x) < 1$ for all
  $x \in \Omega$.

  We show now that $u>0$. As $u$ is continuous in $\overline{\Omega}$,
  it attains its minimum. Assume, by contradiction, that
  $\min_{\overline{\Omega}} u = -m$ for some $m \geq 0$.  Then, by
  applying the maximum principle to $-u$, we know that there exists
  $\xi \in \partial\Omega$ such that $u(\xi) = -m$.  We apply now
  Hopf's Lemma, more precisely its extension from~\cite{oddson} to
  domains satisfying the cone condition (instead of the ball condition
  as in the original work by Hopf, see e.g.~\cite[Theorem~8 in
  Chapter~2]{protter}), which implies that the normal derivative of
  $u$ at $\xi$ is positive, contradicting~\eqref{eq:11}.

  \smallskip

  \noindent$\star$~\ref{it:u3}: is an immediate consequence
  of~\ref{it:u2}, due to the boundary conditions in~\eqref{eq:11}.

  \smallskip

  \noindent$\star$~\ref{it:u4}: denote by $D^2u$ the Hessian matrix of
  $u$ and note that
  \[
      \bigl\{x \in \Omega \colon \nabla u (x)=0\bigr\}
      =
            \bigl\{
            x \in \Omega \colon \nabla u (x)=0 \,\mbox{and}\, \det D^2u (x) = 0
            \bigr\}
          \cup
          \bigl\{
          x \in \Omega \colon \nabla u (x)=0 \,\mbox{and}\, \det D^2u (x) \neq 0
          \bigr\} .
  \]
  The former set has 2-dimensional measure zero by Sard
  Theorem~\cite{sard} applied to $\nabla u$. The latter set consists
  of isolated points all belonging to the compact set
  $\overline{\Omega}$, hence it is finite. Therefore,
  $\modulo{\{x \in \Omega \colon \nabla u (x)=0\}}=0$.

  \smallskip

  \noindent$\star$~\ref{it:u5}: we verify that $u$ is $\C3$ at the
  points in $\Gamma_c$ under condition~\ref{omega3}. To this aim, we
  adapt the arguments in~\cite[Proof of Theorem~3.1]{volkov}, there
  applied to Poisson equation.

  Fix $x_o \in \Gamma_c \cap \overline{\Gamma_e}$, i.e., $x_o$ is a
  doorjamb.  Let $\epsilon$ be as in~\ref{omega3}, call
  $\ell = \epsilon / 2$ and choose $x_1 \in \Gamma_e \cap B(x_o,\ell)$
  with $x_1\neq x_o$.  Let $\nu$ be a unit vector such that
  $\nu \cdot (x_1 - x_o) = 0$ and pointing outward $\Omega$ at
  $x_1$. Define $x_2 = x_1 - \ell \; \nu$ and
  $x_3 = x_o - \ell \; \nu$.  Call $R$ the open rectangle with
  vertexes $x_o$, $x_1$, $x_2$, $x_3$, denote by $x_i\,x_j$ the open
  segment
  \begin{displaymath}
    x_i\,x_j
    =
    \left\{
      x\in\reali^2 \colon
      x = (1-\theta) \; x_i+\theta \; x_j, ~ \theta\in \left]0,1\right[
    \right\}
  \end{displaymath}
  and by $\mathfrak{S}$ the symmetry about the straight line including
  $x_ox_3$ and $R' = \mathfrak{S}(R)$. Define the rectangle
  $\mathcal{R} = R \cup x_ox_3 \cup R'$
  and consider the problem\\
  \begin{minipage}[b]{0.7\linewidth}
    \begin{displaymath}
      \left\{
        \begin{array}{l@{\qquad}r@{\;}c@{\;}l}
          - \delta^2 \; \Delta w(x) + w(x) = 0
          & x & \in & \mathcal{R}
          \\
          w (\xi) = 1
          & \xi & \in & x_ox_1 \cup \mathfrak{S}(x_ox_1)
          \\
          w (\xi) = u(\xi)
          & \xi & \in & x_1x_2\cup x_2x_3
          \\
          w (\xi) = w\left(\mathfrak{S}(\xi)\right)
          & \xi & \in & \mathfrak{S}(x_1x_2\cup x_2x_3) \;.
        \end{array}
      \right.
    \end{displaymath}
  \end{minipage}%
  \begin{minipage}{0.23\linewidth}
    \begin{center}
      \includegraphics[width=\linewidth]{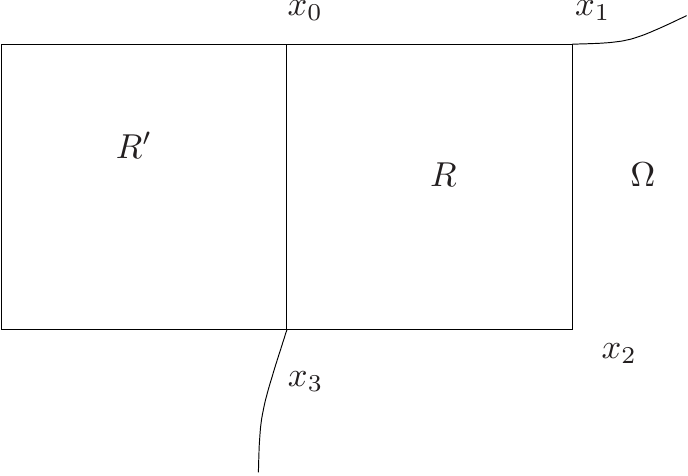}
    \end{center}
  \end{minipage}

  \noindent Note that the boundary condition is of class $\C\infty$ by
  the regularity of $u$ proved above. Lax--Milgram Lemma ensures that
  the function $w$ exists, is unique and is in
  $\C\infty (\mathcal{R}; \reali)$.  By construction, $w$ is symmetric
  with respect to the straight line $x_o + \reali \, \nu$, in the
  sense that
  \begin{displaymath}
    w (x) =  w\left(\mathfrak{S}(x)\right)
    \quad \mbox{ for all } \quad
    x \in \mathcal{R}.
  \end{displaymath}
  This in turn implies that
  \begin{displaymath}
    \nabla w (\xi) \cdot \nu (\xi) = 0
    \quad \mbox{ for all } \quad
    x \in x_ox_3 \;.
  \end{displaymath}
  Due to the $\C\infty$ regularity of the boundary of $\mathcal{R}$ at
  $x_o$, $w$ is of class $\C\infty$ in a neighborhood of $x_o$. By
  uniqueness, $w = u$ on $\overline{\mathcal{R}}$. Hence, $u$ is of
  class $\C\infty$ also in a neighborhood of $x_o$ restricted to
  $\Omega$.

  If $x_o \in (\Gamma_c \setminus \overline{\Gamma_e})$, to prove the
  regularity of $u$ at $x_o$ we proceed as above, simply replacing the
  Dirichlet condition on $x_ox_1$ by a homogeneous Neumann one,
  applying again Lax-Milgram Lemma and concluding by symmetry and
  uniqueness.

  \smallskip

  \noindent$\star$~\ref{it:u6}: the characteristic equation
  $\det \left(D^2 u(\bar{x}) - \lambda I \right) = 0$ in the case of a
  2-dimensional problem is a quadratic equation with real solutions
  $\lambda_1(\bar{x})$, $\lambda_2(\bar{x})$ satisfying
  \begin{displaymath}
    \lambda_1(\bar{x}) \; \lambda_2(\bar{x}) = \det D^2 u(\bar{x}) \;,
    \qquad
    \lambda_1(\bar{x}) + \lambda_2(\bar{x}) = \Delta u(\bar{x}) \;.
  \end{displaymath}
  Note that by the $\C2$ regularity of $u$ proved at~\ref{it:u1}, the
  equation $u = \delta^2\Delta u$ is satisfied in whole
  $\overline{\Omega}$.  By~\ref{it:u2},
  $\lambda_1(\bar{x}) + \lambda_2(\bar{x}) = \delta^{-2} \, u(\bar{x})
  > 0$,
  so that at least one of the eigenvalues has to be (strictly)
  positive.
\end{proof}

\begin{proofof}{Proposition~\ref{prop:1}}
  By~\eqref{eq:trans} and straightforward computations it is clear
  that~\eqref{eq:ell} has a solution if and only if~\eqref{eq:11} has
  a solution which is positive a.e.~in $\Omega$.  Point~\ref{it:u1} in
  Lemma~\ref{lem:eu} ensures the existence and uniqueness of a
  solution to~\eqref{eq:11}. Moreover, by~\ref{it:u2} in
  Lemma~\ref{lem:eu} this solution is strictly positive a.e.~in
  $\Omega$. This allows to define $\phi = -\delta \ln u$. The
  remaining regularity statements and~\ref{prop:1.1} follow again from
  Lemma~\ref{lem:eu} by~\eqref{eq:trans}. So as to
  obtain~\ref{prop:1.2}, note first that
  $-\grad\phi \cdot \nu = \frac{\delta}{u} \, \grad u \cdot \nu > 0$
  everywhere on $\Gamma_e$ by~\eqref{eq:trans} and~\ref{it:u3} in
  Lemma~\ref{lem:eu}. Then, integrate~\eqref{eq:11} on~$\Omega$, use
  Green Theorem and again Lemma~\ref{lem:eu} to obtain~\ref{prop:1.3}.
\end{proofof}

\begin{proofof}{Proposition~\ref{prop:3}}
  The present proof follows from~\cite[Theorem~2.7]{ColomboRossi2015}.
  Indeed, referring to the notation therein, we define
  $q (\rho) = \rho \, v (\rho)$ and verify the necessary assumptions.
  \begin{description}
  \item[($\mathbf{\Omega_{3,\gamma}}$)] $\Omega$ is a bounded open
    subset of $\reali^2$ with piecewise $\C{3,\gamma}$ boundary
    $\partial\Omega$ by~\ref{omega1} and~\ref{omega2}.
  \item[(F)] This condition is immediate since in the present case we
    have $F\equiv0$.
  \item[(f)] In our case $f (t,x,\rho) = \rho \, v (\rho) \, w (x)$.
    By~\ref{V} and the assumption that $w$ is in
    $(\C2 \cap \W1\infty) \left(\reali;B(0,1)\right)$, we have that
    $f$ is of class $\C2$ and moreover
    \begin{align*}
      &\partial_\rho f (t,x,\rho)
       =
       q' (\rho) \, w (x) \;,&
      & \partial^2_{\rho\rho} f (t,x,\rho)
       =
       q'' (\rho) \, w (x) \;,&
      & \partial_\rho \div f (t,x,\rho)
       =
       q' (\rho \,) \div w (x)
    \end{align*}
    are all functions of class $\L\infty$ on
    $\reali^+ \times \Omega \times [0,R_{\max}]$.
  \item[(C)] This condition follows from~\ref{C} because in the
    present case $\rho_b\equiv0$.
  \end{description}
  We then obtain
  \begin{displaymath}
    \begin{array}{@{}r@{\,}c@{\,}l@{\quad}l@{}}
      \norma{\mathcal{S}_t\rho_o}_{\L\infty(\Omega;\reali)}
      & \leq
      & \left( \norma{\rho_o}_{\L\infty(\Omega;\reali)} + c_2 \, t \right)
        \exp(c_1 \, t)
      & \mbox{by~\cite[Formula~(2.5)]{ColomboRossi2015}}
      \\
      \tv(\mathcal{S}_t\rho_o)
      & \leq
      & \left(
        \mathcal{A}_1 + \mathcal{A}_2 \, t + \mathcal{A}_3 \, \tv(\rho_o)
        \right)
        \exp(\mathcal{A}_4\,t)
      & \mbox{by~\cite[Formula~(6.44)]{ColomboRossi2015}}
    \end{array}
  \end{displaymath}
  where, with reference to~\cite[Formula~(5.1)]{ColomboRossi2015}
  and~\cite[\S~6]{ColomboRossi2015}, the constants $c_1$, $c_2$,
  $\mathcal{A}_1$, $\ldots$, $\mathcal{A}_4$ are estimated as follows:
  \begin{align*}
    &c_1
     =
     1
      +
      \norma{q'}_{\L\infty([0,R_{\max}];\reali)} \, \norma{\nabla\cdot w}_{\L\infty(\Omega;\reali)}
    \leq
      1
      +
      \norma{q}_{\W1\infty([0,R_{\max}];\reali)} \,
      \norma{w}_{\W1\infty(\Omega;\reali)} \;,
    \\
    &c_2
     =
     0 \;,
    \\
    &\mathcal{A}_1
     =
     \mathcal{O}(1) \, \norma{D f)}_{\L\infty(\Omega\times[0,R_{\max}];\reali^{n\times(1+n)})}
     \leq
     \mathcal{O}(1) \, \norma{q}_{\W1\infty([0,R_{\max}];\reali)} \,
      \norma{w}_{\W1\infty(\Omega;\reali^n)} \;,
    \\
    &\mathcal{A}_2
     =
     \mathcal{O}(1) \,
     \norma{D f}_{\W1\infty(\Omega\times[0,R_{\max}];\reali^{n\times(1+n)})}
      \leq
      \mathcal{O}(1) \, \norma{q}_{\W2\infty([0,R_{\max}];\reali)} \,
      \norma{w}_{\W2\infty(\Omega;\reali^n)} \;,
    \\
    &\mathcal{A}_3
     =
     \mathcal{O}(1)
      +
      \norma{q'}_{\L\infty([0,R_{\max}];\reali)} \, \norma{w}_{\L\infty(\Omega;\reali^n)}
     \leq
     \mathcal{O}(1)
      +
      \norma{q}_{\W1\infty([0,R_{\max}];\reali)} \,
      \norma{w}_{\L\infty(\Omega;\reali^n)} \;,
    \\
    &\mathcal{A}_4
     =
     \mathcal{O}(1)
      \left[
      1
      +
      \norma{D f}_{\W1\infty(\Omega\times[0,R_{\max}];\reali^{n\times (1+n)})}
      \right]
     \leq
     \mathcal{O}(1)
      \left[
      1
      +
      \norma{q}_{\W2\infty([0,R_{\max}];\reali)} \, \norma{w}_{\W2\infty(\Omega;\reali^n)}
      \right]
  \end{align*}
  and the above norms of $q$ are bounded by~\ref{V} and by the adopted
  assumption on $w$.
\end{proofof}

For technical reasons, below we fix an arbitrary open subset $\Omega'$
of $\reali^2$ containing $\overline{\Omega}$ and extend the unique
generalized solution $\phi \in \C3 (\overline\Omega; \reali)$
of~\eqref{eq:ell} given in Proposition~\ref{prop:1} introducing a map
$\widetilde{\phi} \in \Cc3\left(\reali^2; \reali \right)$ such that
$\widetilde{\phi} \equiv \phi$ in $\Omega$ and
$\widetilde{\phi} \equiv 0$ in $\reali^2 \setminus \Omega'$.  This is
possible thanks to the regularity of $\phi$ and to the following
result.

\begin{lemma}[{\cite[Lemma~6.37]{gilbarg}}]
  \label{lem:GT}
  Let $\Omega$ satisfy~\ref{omega1}, \ref{omega2}, \ref{omega3}. For
  any open subset $\Omega'$ of $\reali^2$ such that
  $\overline{\Omega} \subset \Omega'$, there exists a constant $C$
  such that for any $f \in \C3 (\Omega; \reali)$, there exists a map
  $\widetilde{f} \in \Cc3 (\reali^2; \reali)$ with
  \begin{displaymath}
    \widetilde{f} (x) =
    \left\{
      \begin{array}{l@{\quad\mbox{for all }x \in\;}l}
        f (x) & \Omega
        \\
        0 & \reali^2 \setminus \Omega'
      \end{array}
    \right.
    \quad\mbox{ and }\quad
    \norma{\widetilde{f}}_{\C3
      (\reali^2; \reali)} \leq C \, \norma{f}_{\C3
      (\overline{\Omega}; \reali)} \;.
  \end{displaymath}
\end{lemma}

\begin{proofof}{Proposition~\ref{prop:path}}
  First, apply Lemma~\ref{lem:GT} and extend $\phi$ to a
  $\widetilde\phi \in \C3 (\reali^2; \reali)$.

  Define
  $\widetilde w (x) = \mathcal{N}\left(-\nabla \widetilde{\phi}
    (x)\right)$.
  By~\eqref{eq:Nprop}, Lemma~\ref{lem:GT} and
  Proposition~\ref{prop:1},
  $\widetilde w \in \C{0,1} (\reali^2; \reali^2)$.  Hence, for any
  fixed $\hat x \in \reali^2$, the Cauchy problem
  \begin{align}
  \label{eq:tildeW}
      &\dot x
        =
        \widetilde w (x)\;,&
      &x (0) = \hat x
  \end{align}
  admits a unique solution
  $\widetilde p_{\hat x} \colon \reali \to \reali^2$. Define
  \begin{displaymath}
    T_{\hat x}
    =
    \sup
    \left\{
      t \in \reali^+ \colon \widetilde p_{\hat x} ([0,t]) \subset \Omega
    \right\}
    \quad \mbox{ and } \quad
    p_{\hat x} (t)
    =
    \widetilde p_{\hat x} (t)
    \quad \mbox{ for } t \in [0, T_{\hat x}] \;.
  \end{displaymath}
  By construction, the map $p_{\hat x}$ solves \eqref{eq:pat}.  By the
  standard theory of ordinary differential equations, \ref{quark1}
  and~\ref{quark2} are proved.

  We consider now~\ref{quark3}.  Note that~\eqref{eq:tildeW} is
  dissipative in $\Omega$, in the sense that $\widetilde{\phi}$ is a
  (strict) Lyapunov function for~\eqref{eq:tildeW} in $\Omega$, i.e.,
  $\widetilde{\phi}$ decreases along the path $t \to p_{\hat x} (t)$
  as long as $p_{\hat x} (t) \in \Omega$. In fact, as long as
  $p_{\hat x} (t) \in \Omega$
  \begin{displaymath}
    \frac{\d~}{\d{t}} \; \widetilde{\phi} \left(p_{\hat x} (t) \right)
    =
    \frac{\d~}{\d{t}} \; \phi \left(p_{\hat x} (t) \right)
    =
    -
    \left(
      \theta^2
      +
      \norma{\grad \! \phi \left(p_{\hat x} (t) \right)}^2
    \right)^{-1/2}
    \;
    \norma{\grad \! \phi \left(p_{\hat x} (t)\right)}^2 \,,
  \end{displaymath}
  which is strictly negative whenever $\hat x$ is not a critical
  point.  By La Salle Principle~\cite[Theorem~9.22, see also Lemma
  9.21 and Theorem~14.17]{hale_kocak}, as $t$ goes to infinity, every
  bounded path $p_{\hat x}$ that remains in $\Omega$ is attracted
  towards the set of equilibria, i.e., of critical points
  of~\eqref{eq:tildeW}. More precisely, setting
  \begin{align*}
    &\omega(\hat x)
     =
     \left\{
      x\in\reali^2 \colon
      \begin{array}{l}
        \mbox{there exists } (t_n)_{n\in\naturali}
        \mbox{ such that}
        \\
        \lim\limits_{n\to\infty} t_n = \infty \ \mbox{ and }
        \lim\limits_{n\to\infty} p_{\hat x}(t_n) = x
      \end{array}
    \right\},&
    &{\cal E}_D
     =
     \left\{
      x \in D \colon \grad\widetilde{\phi}(x) = 0
      \right\}
      \quad \mbox{for } D\subseteq \reali^2
  \end{align*}
  we proved that if $x \in \omega (\hat x) \cap \Omega$ for a
  $\hat x \in \Omega$, then $\nabla\phi (x) = 0$.

  Note that for any $\hat x \in \Omega$, the path $\tilde p_{\hat x}$
  exiting $\hat x$ does not intersect $\Gamma_w$. Indeed, by the
  boundary condition imposed along $\Gamma_w$ in~\eqref{eq:ell}
  \begin{displaymath}
    \Gamma_w
    =
    \left\{
      x \in \Gamma_w \colon \nabla \phi (x) = 0
    \right\}
    \cup
    \left\{
      x \in \Gamma_w \colon \nabla\phi (x)  \neq 0
      \mbox{ and }
      \nabla\phi (x) \cdot \nu (x) = 0
    \right\} \,.
  \end{displaymath}
  The former set above is clearly invariant, both positively and
  negatively, with respect to~\eqref{eq:tildeW}, hence it can not be
  reached by a path $t \to p_{\hat x} (t)$ starting in $\Omega$. The
  latter consists of trajectories solving~\eqref{eq:tildeW} that are
  entirely contained in $\Gamma_w$, since $w$ is parallel to
  $\Gamma_w$. As a consequence, for any $\hat x \in \Omega$, either
  the path $t \to p_{\hat x} (t)$ crosses $\Gamma_e$, or it stays in
  $\Omega$ and approaches a point in the set
  ${\cal E}_{\overline{\Omega}}$, namely
  $\omega(\hat x) \subseteq {\cal E}_{\overline{\Omega}}$.

  It remains to determine the behaviour of the system near the
  critical points in ${\cal E}_{\overline{\Omega}}$.  We proceed by
  linearisation around $\bar{x}$, with $\nabla\phi (\bar x) = 0$.
  Denote by $A(\bar{x})$ the first order total derivative of
  ${\cal N}(-\nabla\phi)$ computed at
  $\bar{x}\in{\cal E}_{\overline{\Omega}}$. By direct computations,
  \begin{equation}
    \label{eq:uffa}
    A(\bar{x})
    =
    D{\cal N} \left(-\nabla\phi(\bar{x})\right)
    = - \theta^{-1} D^2\phi(\bar{x}) \;,
  \end{equation}
  thanks to $\nabla\phi (\bar x) = 0$. Recall the map $u$ given
  by~\eqref{eq:trans}. Due to~\eqref{eq:11} and~\eqref{eq:uffa} we
  have
  \begin{displaymath}
    A(\bar{x})
    =
    \frac{1}{\theta}\,
    \frac{\delta}{u(\bar{x})} ~ D^2 u(\bar{x}) \;,
  \end{displaymath}
  proving that $A (\bar x)$ is symmetric and diagonalizable.
  By~\ref{it:u6} in Lemma~\ref{lem:eu}, $A (\bar u)$ has at least one
  strictly positive eigenvalue, say $\lambda_2>0$.  Consider now two
  cases, depending on the value attained by the other eigenvalue
  $\lambda_1$:

  \smallskip

  \noindent%
  \textbf{$\star \quad\lambda_1\neq0$:} Then, by Hartman-Grobman
  Theorem, see e.g.~\cite[Theorem~9.35]{hale_kocak}, depending on the
  sign of $\lambda_1$, $\bar{x}$ is either a source or a saddle. In
  both cases, it is an isolated point of
  ${\cal E}_{\overline{\Omega}}$, so that $\bar{x} \in \omega(\hat x)$
  implies $\{\bar{x}\} = \omega(\hat x)$, by the connectedness of
  $\omega(\hat x)$. This is possible only if $\lambda_1 < 0$, i.e.,
  $\bar{x}$ is a saddle, and $\hat x$ belongs to the stable manifold
  consisting of two trajectories entering $\bar x$, which is a set of
  measure zero.

  \smallskip

  \noindent%
  \textbf{$\star \quad\lambda_1=0$:} Then, $\bar{x}$ is not
  necessarily an isolated point of $\mathcal{E}_{\overline \Omega}$.
  We use here the result of Palmer~\cite{palmer} about the local
  central manifold, which is an invariant 1-dimensional set containing
  all possible critical points in a neighborhood of $\bar{x}$. This
  result can be seen as a generalization of the Hartman-Grobman
  Theorem, and gives the instability of the central manifold, see
  also~\cite[\S~4]{aulbach}, \cite[\S~9.2-9.3]{chow_hale},
  \cite[Theorem~10.14]{hale_kocak}.

  Let $B$ be the change of coordinates matrix such that
  $B\, A (\bar x) \, B^{-1}$ is diagonal, with $A (\bar x)$ given
  in~\eqref{eq:uffa}. By means of the linear change of variables
  $y(t) \!=\! B \, \left(p_{\hat x}(t) \!-\! \bar{x}\right)$, the differential
  equation in~\eqref{eq:tildeW} can be written as
  \begin{align}
  \label{eq:f}
    &\dot{y}_1 = f_1(y_1,y_2)\;,&
    &\dot{y}_2 = \lambda_2 \, y_2 + f_2(y_1,y_2) \;,
  \end{align}
  where $f\in\C2(\reali^2; \reali^2)$ is bounded, see
  Lemma~\ref{lem:GT}, and satisfies $f(0) = 0$.  The dependence of
  $B$, $f$ and $\lambda_2$ upon $\bar{x}$ is here neglected.  We
  obtain from~\cite{palmer} that there exist a Lipschitz continuous
  function $h$ and a homeomorphism
  $H \colon \reali^+\times\reali^2\to\reali^2$, such that the graph of
  $h$ is the local central manifold and the map
  $z(t) = H\left(t,y(t) \right)$, with $H (t,0) = 0$, solves
  \begin{align}
  \label{eq:z}
      &\dot{z_1} = f_1(z_1,h(t,z_1))\;,&
      &\dot{z_2} = \lambda_2 \, z_2 \;,
  \end{align}
  provided $y$ solves~\eqref{eq:f}. As a matter of fact, $h$ can be
  proved to be also $\C2$, see~\cite[Proposition~4.1]{aulbach}
  or~\cite[Theorem~10.14]{hale_kocak}.

  Then, by continuity of $H$, there exists $r_0 > 0$ such that if
  $\norma{y(t)}<r_0$, then
  $\modulo{z_2(t)} = \modulo{H_2\left(t,y(t) \right)} <
  \modulo{z_{2}(0)}$.
  Solving the second equation in~\eqref{eq:z}, we obtain that for
  $y(0)$ such that $z_2 (0) = H_2\left(0, y(0) \right)\neq 0$, there
  exists $t_* >0$ such that $\norma{y(t)}>r_0$ for all $t>t_*$.  Going
  back to the original $x$-variable, for any neighborhood ${\cal O}$
  of $\bar x$ with ${\cal O} \subseteq \reali^2$, introduce
  $
    W
    =
    \{
      x \in{\cal O}\colon  H_2\left(0,B(x-\bar{x})\right) = 0
    \} \,.
  $
  We have obtained that if $\!\hat{x}\! \in \!{\cal O}\!\setminus\! W\!$, then
  $\!p_{\hat{x}}(t)\!$ is outside $\!\mathcal{O}\!$ for all $t \!>\! t_*$.  Thus,
  $\!\bar{x}\!$ can be attractive only for the points lying on $W$, which
  is clearly a \!1-dimensional manifold and has 2-dimensional Lebesgue
  measure equal to $\!0$. Moreover, $W\!$ as a whole is repulsive.

  Therefore, $\omega(\hat x)\cap W$ is non-empty only if the path
  passing through $\hat x$ lies inside $W$. Therefore, the
  $1$-dimensional Lebesgue measure of $\omega(\hat x)\cap W$ is $0$.


  Finally, for almost all $\hat{x}$, the path $p_{\hat{x}}(\reali^+)$
  given by~\eqref{eq:tildeW} is not attracted by
  ${\cal E}_{\overline{\Omega}}$, hence it has to reach the exit
  $\Gamma_e$, i.e., there exists a positive finite time $T_{\hat{x}}$
  such that $p_{\hat x}(T_{\hat{x}}) \in \Gamma_e$.
\end{proofof}
\noindent\textbf{Acknowledgment:} The authors were supported by the
INDAM-GNAMPA project \emph{Leggi di conservazione nella
  modellizzazione di dinamiche di aggrega\-zio\-ne}. The last author
was partially supported by ICM, UW.

\def\cprime{$'$}

\end{document}